\documentclass[10pt,english,letter, reqno]{amsart}

\usepackage{amssymb,url,xspace,amsthm,adjustbox}

\usepackage[bibstyle=authoryear, citestyle=authoryear]{biblatex}
\bibliography{refs.bib}
\usepackage{datetime}
\usepackage[T1]{fontenc}
\usepackage{chngcntr}
\usepackage{graphicx}
\usepackage{mathtools}
\usepackage{color}
\usepackage{amsmath}
\usepackage{setspace}
\usepackage{mathrsfs,stmaryrd}
\usepackage{yfonts, bbm}
\usepackage{enumitem}

\usepackage{hyperref}
\hypersetup{
  colorlinks   = true, 
  urlcolor     = blue, 
  linkcolor    = blue, 
  citecolor   = green 
}

\usepackage{todonotes}

\usepackage{float}

\theoremstyle{plain}
      \newtheorem{theorem}{Theorem}

      \newtheorem{lemma}[theorem]{Lemma}
      \newtheorem{corollary}[theorem]{Corollary}

      \theoremstyle{definition}

      \newtheorem{remark}[theorem]{Remark}

      \newtheorem*{conjecture*}{Conjecture}

\author[A. Fiori]{Andrew Fiori}
\address{Department of Mathematics and Statistics, University of Lethbridge,
4401 University Drive,
Lethbridge, Alberta,
T1K 3M4,
Canada}
\email{andrew.fiori@uleth.ca}
\thanks{Andrew Fiori thanks and acknowledges the University of Lethbridge for their financial support as well as the support of NSERC Discovery Grant RGPIN-2020-05316.}

\title{A Note on the Phragm\'en-Lindel\"of Theorem}
\date{\today}                                           
\begin{document}
\begin{abstract}
    We provide a generalization of the Phragm\'en-Lindel\"of principal of Rademacher with the aim  of correcting, or at least provide a pathway to correcting, several errors appearing in the literature.
\end{abstract}

\maketitle

The Phragm\'en-Lindel\"of principle is well known and widely used in analytic number theory as a tool to obtain bounds on $L$-functions by interpolating bounds obtainable on vertical lines.
One of the early general purpose results which is now widely used is \cite[Theorem 2]{Rademacher}. There is a generalization of this result presented in \cite[Lemma 3]{Tr14-2}. The proof of Rademacher is quite technical as it requires the construction of an analytic function with very precisely described growth properties. The proof of Trudgian contains an error\footnote{the function $F(s)$ defined in the proof is not in general analytic as it is a function of $\sigma={\rm Re}(s)$.} which arises in trying to avoid doing the same technical argument with power of $\log(Q+s)$ rather than $(Q+s)$.

In what follows we sufficiently correct the error so that proofs using it should be recoverable, at worst at the expense of some constants. Moreover, we provide a generalized version Rademacher's work using what is in our view a simpler proof than that of Rademacher. 
We note that the result of Trudgian has been widely cited (45 direct citations, many of which have even more), including several results which now improve upon Trudgian's main result (see \cite{HSW,BW}). Consequently, that our result in practice recovers (at the expense of extra checks for small values, see Section \ref{sec:incdec}) and improves upon Trudgian's work (see Sections \ref{sec:errors} and \ref{sec:examples}) means our work shall in the worst case allow one to almost recover these other result.

In Section \ref{sec:simple} we present the simple form of our main theorem and its proof. In Section \ref{sec:incdec} we present a slightly more complicated version of the theorem which is more adaptable to many real world settings.
In Section \ref{sec:errors} we address a common error in the literature and illustrate how one can recover, within a very small margin, the values that would arise from making this error while simultaneously illustrating how to improve on what would normally have been the correct use of the theorem. Finally, in Section \ref{sec:examples} we give several examples of applications of our theorems to the Riemann zeta function. In these examples we recover, within a factor of $(1+10^{-9})$ for $t>10^{5}$, and in many cases improve upon the erroneous applications we are aware of in the literature. 

\section{Simple form of Main Theorem}\label{sec:simple}

\begin{theorem}[Phragm\'en-Lindel\"of principle]\label{thm:phrag}
Let $a$ and $b$ be real numbers and let $G_1,\ldots, G_r$ be complex functions that are holomorphic for ${\rm Re}(s)\in [a,b]$, have $\overline{G_i(s)}=G_i(\overline{s})$, and satisfy the following growth condition for all $\sigma\in[a,b]$
\[ C_1\exp \left(-C_2 e^{C_3|t|}\right) < |G_i(\sigma + it)|\]
for some  $C_1,C_2>0$ and $0<C_3<\frac{\pi}{b-a}$.
Suppose that $f$ is a holomorphic function for ${\rm Re}(s)\in [a,b]$ and satisfies the growth condition
\[
|f(\sigma + it)|<C_1\exp \left(C_2 e^{C_3|t|}\right)
\]
for $\sigma\in[a,b]$.
 Suppose that we have real numbers $\alpha_1,\ldots, \alpha_r\ge 0$ and $\beta_1,\ldots,\beta_r \ge 0$ such that
\[
|f(a+it)| \leq
\prod_{i=1}^r \left|G_i(a+it) \right|^{\alpha_i}  \]
and
\[
|f(b+it)| \leq
\prod_{i=1}^r \left| G_i(b+it)\right|^{\beta_i}.  \]
Finally, suppose that for $\sigma\in[a,b]$ and for all $t$ the function $|G_i(\sigma+it)|$ is increasing in $\sigma$ if $\alpha_i>\beta_i$ and decreasing if $\beta_i>\alpha_i$.
Then for $a \leq \sigma \leq b$, we have
\[|f(\sigma+it)| \leq
\bigg(\prod_{i=1}^r \left|G_i(\sigma+it)\right|^{\alpha_i}   \bigg)^{\frac{b-\sigma}{b-a}}
\bigg(\prod_{i=1}^r \left|G_i(\sigma+it)\right|^{\beta_i}  \bigg)^{\frac{\sigma-a}{b-a}} .
\]
\end{theorem}

\begin{remark}
    The result can be extended to have $\alpha_i$ and $\beta_i$ negative at the expense of adding the additional growth condition
    \[ |G_i(\sigma + it)| < C_1\exp \left(C_2e^{{C_3}|t|}\right) \]
for some  $C_1,C_2>0$ and $0<C_3<\frac{\pi}{b-a}$ for $\sigma\in[a,b]$ as a hypothesis.  The conclusion is otherwise identical. The modified proof is to simply multiply $f(s)$ by $G_i(s)^{-\min(\alpha_i,\beta_i)}$.

Some standard choices for $G_i(s)$ are
\begin{enumerate}
\item $G(s) = e$ where $e^\alpha=A$ and $e^\beta=B$.
\item $G(s) = (Q+s)$ where $Q>-a$.
\end{enumerate}
\end{remark}

\begin{corollary}\label{cor:simple}
Let $a$ and $b$ be real numbers and suppose $f(s)$ is holomorphic for ${\rm Re}(s)\in [a,b]$ and satisfies the following growth condition for all $\sigma\in[a,b]$
\[
|f(\sigma + it)|<C_1\exp \left(C_2e^{{C_3}|t|}\right)
\]
for some $C_1,C_2>0$ and $0<C_3<\frac{\pi}{b-a}$.

Suppose there exist $A,B,\alpha_1 \ge \beta_1, \alpha_2 = \beta_2, \alpha_3 = \beta_3,Q_1>-a,Q_2>1-a,Q_3>e-a$ such that:
\[ |f(a+it)| \leq A\big|Q_1+a+it\big|^{\alpha_1}\big|\log(Q_2+a+it)\big|^{\alpha_2}\big|\log\log(Q_3+a+it)\big|^{\alpha_3} \]
and
\[ |f(b+it)| \leq B\big|Q_1+b+it\big|^{\beta_1}\big|\log(Q_2+b+it)\big|^{\beta_2}\big|\log\log(Q_3+b+it)\big|^{\beta_3} \]
then if $\sigma\in [a,b]$ we have $|f(\sigma+it)|$ bounded by
\begin{align*}\bigg( A\big|Q_1&+\sigma+it\big|^{\alpha_1}\big|\log(Q_2+\sigma+it)\big|^{\alpha_2}\big|\log\log(Q_3+\sigma+it)\big|^{\alpha_3}  \bigg)^{\frac{b-\sigma}{b-a}}\cdot\\&
\bigg( B\big|Q_1+\sigma+it\big|^{\beta_1}\big|\log(Q_2+\sigma+it)\big|^{\beta_2}\big|\log\log(Q_3+\sigma+it)\big|^{\beta_3}  \bigg)^{\frac{\sigma-a}{b-a}} .
\end{align*}
\end{corollary}
\begin{proof}
The functions $G_1(s)=e$, $G_2(s)=Q_1+s$, $G_3(s)=\log(Q_2+s)$ and $G_4(s)=\log\log(Q_3+s)$ will satisfy all the hypotheses of the theorem given that $\alpha_2=\beta_2$ and $\alpha_3=\beta_3$.
\end{proof}

\begin{remark}
        The function $|\log (Q+ s)|$ is neither always increasing nor always decreasing in $\sigma$. The behavior depends on $t$ and $\sigma$. For this reason the argument we gave above only works with $\alpha_2=\beta_2$ and $\alpha_3=\beta_3$.

Moreover, you should recall that $|\log(Q_2+s)| \ge \log|Q_2 + s|$ and
    $|\log\log(Q_2+s)| \ge \log\log|Q_2 + s|$.
    Indeed 
    \[ |\log(Q_2+s)| = \bigg(1+\frac{{\rm arg}(Q_2+s)^2}{(\log|Q_2+s|)^2}\bigg)^{1/2}(\log|Q_2+s|)\]
    and
    \[ |\log\log(Q_2+s)| = \bigg(1+\frac{{\rm arg}(\log(Q_2+s))^2}{(\log|\log(Q_2+s)|)^2}\bigg)^{1/2}
    \log|\log(Q_2+s)|
    \]
    where
    \[ {\rm arg}(\log(Q_2+s))
        =\arctan\bigg(\frac{{\rm arg}(Q_2+s)}{\log|Q_2+s|}\bigg).
    \]
    This difference can be relevant in both the hypothesis and conclusion of the theorem depending on the signs of the exponents. There are multiple results in the literature which fail to account for this. See Sections \ref{sec:errors} and \ref{sec:examples} for a method to improve this.
\end{remark}

\begin{lemma}\label{lem:bisect}
    With the hypothesis as in Theorem \ref{thm:phrag} we have
    \[ |f((a+b)/2 + it)| < \left| \prod_{i=1}^r G_i((a+b)/2 +it)^{\alpha_i+\beta_i} \right|^{1/2}. \]
    That is,  Theorem \ref{thm:phrag} holds where $\sigma$ is the midpoint between $a$ and $b$.
\end{lemma}
\begin{proof}
We shall denote by $\overline{f}$ the function $\overline{f}(s) = \overline{f(\overline{s})}$ which is analytic for ${\rm Re}(s)\in[a,b]$.
We note that a bound on $|f(s)|$ in terms of $G_i(s)$ satisfying $\overline{G_i(s)} = G_i(\overline{s})$ implies the same bound on $|\overline{f}(s)|$.

Consider the function
\[ F(s) = f(a+s)\overline{f}(b-s).\]
Define
\[ H_i(s) = \begin{cases} 
G_i(a +s)^{\alpha_i}G_i(b-s)^{\beta_i} & \alpha_i > \beta_i 
\\
G_i(a +s)^{\beta_i}G_i(b-s)^{\alpha_i} & \alpha_i < \beta_i 
\\
G_i(a +s)^{\alpha_i}G_i(b-s)^{\beta_i} & \alpha_i=\beta_i
\end{cases}
\]
and
\[ H(s) = \prod_{i=1}^r H_i(s).
\]
The growth conditions in $G_i$ imply they are non-vanishing for ${\rm Re}(s)\in[a,b]$ which allows us to conclude that
\[ F(s)/H(s) \]
is analytic in $[0,(b-a)]$ and it satisfies the necessary growth condition to apply the standard Phragm\'en-Lindel\"of principle (see for example \cite[Theorem 12.9]{Rudin}). Moreover, for ${\rm Re}(s)=0$ and ${\rm Re}(s)=(b-a)$ we have $F(s)/H(s) \leq 1$ as
\[ F(it) = |f(a+it)\overline{f}(b-it)| = |f(a+it)f(b+it)| \]
and
\[ |F((b-a)+it)|  =  |f(b+it)\overline{f}(a-it)|  = |f(a+it)f(b+it)|\]
are from our hypotheses bounded by
\[ \prod_{i=1}^r |G_i(a+it)|^{\alpha_i}|G_i(b+it)|^{\beta_i}.\] 
Now notice that $|H_i(it)|$ and $|H_i((b-a) + it)| $ take on the values\footnote{which is taken on at which point depends on the case we are considering.}
\[  |G_i(a+it)|^{\alpha_i}|G_i(b+it)|^{\beta_i} \qquad\text{and}\qquad  |G_i(a+it)|^{\beta_i}|G_i(b+it)|^{\alpha_i}. \]
Under the hypotheses the second is larger than the first and the first is the needed bound.

Thus by the Phragm\'en-Lindel\"of principle we have $F(s)/H(s) \leq 1$ on the whole interval.
It follows by evaluating at $s=(b-a)/2 + it$ that for all $t$ we have
\[ |F((b-a)/2+it)| \leq |H((b-a)/2+it)| = \prod_{i=1}^r |G_i((a+b)/2 + it)|^{\alpha_i+\beta_i}. \]
Since
\[ |F((b-a)/2+it)| = |f((a+b)/2 + it)\overline{f}((a+b)/2 - it)| =  |f((a+b)/2 + it)|^2 \]
we obtain the result.
\end{proof}

\begin{proof}[Proof of Theorem \ref{thm:phrag}]
We first claim that the result holds for all $\sigma$ of the form
\[ a + (b-a)\frac{n}{2^k}\]
where $n$ and $k$ are non-negative integers so that $\frac{n}{2^k}\in [0,1]$.
The proof of this proceeds by induction on $k$, using Lemma \ref{lem:bisect} applied to the interval $a=\tfrac{\lfloor n/2 \rfloor}{2^{k-1}}$ and $b=\tfrac{\lceil{n/2}\rceil}{2^{k-1}}$ with the bounds obtained from the inductive hypothesis.

The final result now follows by the continuity of $f(s)$ and $G_i(s)$.
\end{proof}

\begin{remark}
We note that in most applications as functions of $t$ we will have $|G_i(a+it)| \sim |G_i(\sigma+it)| \sim |G_i(b+it)|$ so that if one wants to work with a function where $G_i$ which is decreasing (respectively increasing) in $\sigma$ rather than increasing (respectively decreasing) one can replace $G_i(s)$ by $G(b-a-s)$ at the expense of possibly needing to multiply by a constant to preserve the inequalities on the boundary.
\end{remark}

\section{Dealing with the Condition of Increasing/Decreasing in $\sigma$}\label{sec:incdec}

Even though the function $|\log (Q+ s)|$ is neither always increasing nor always decreasing in $\sigma$ it is desirable to be able to use it as $G_i(s)$. There are options for dealing with such functions:
\begin{itemize}
\item Replace $G(s)$ with $G(Q'-s)$ provided the function was at least consistently always increasing or decreasing.  This is essentially what occurs in \cite[Theorem 2]{Rademacher} for situations $\beta>\alpha$ applied to $G(s) = Q+s$. Notice that the function $|Q-\sigma-it|$ would be decreasing in $\sigma$ when $Q-\sigma>0$ whereas $|Q+\sigma+it|$ is increasing in $\sigma$ where $Q+\sigma>0$.
\item Replace $G(s)$ with $e^{c(s-a)}G(s)$ for some choice $c$ such that the resulting function is appropriately increasing/decreasing.
\item Grouping together several $G_i(s)$ so that the product satisfies the conditions.
\item Satisfy the increasing condition for $t>T_0$ and impose an alternative condition for $t<T_0$.
\end{itemize}
In each case we risk making the final result worse by a constant in order to satisfy the conditions.
We provide a precise version of the last option below.

\begin{theorem}\label{thm:phrag2}
Let $a, b$ be real numbers, and let $G_1,\ldots, G_r$ be complex functions that are holomorphic for ${\rm Re}(s)\in [a,b]$, have $\overline{G_i(s)}=G_i(\overline{s})$, and satisfy the following growth condition for all $\sigma\in[a,b]$
\[ C_1\exp \left(-C_2 e^{{C_3}|t|}\right) < |G_i(\sigma + it)|\]
for some  $C_1,C_2>0$ and $0<C_3<\frac{\pi}{b-a}$.
Suppose that $f$ is a holomorphic function on ${\rm Re}(s)\in [a,b]$ and satisfies the growth condition
\[
|f(\sigma + it)|<C_1\exp \left(C_2 e^{{C_3}|t|}\right)
\]
for all $\sigma\in[a,b]$.
Suppose that we have real numbers $\alpha_1,\ldots, \alpha_r\ge 0$ and $\beta_1,\ldots,\beta_r \ge 0$ such that
\[
|f(a+it)| \leq
\prod_{i=1}^r  \left|G_i(a+it)\right|^{\alpha_i}   \]
and
\[
|f(b+it)| \leq
\prod_{i=1}^r \left| G_i(b+it)\right|^{\beta_i} .  \]
Suppose for $|t| > T_0$ and $\sigma\in[a,b]$ that $|G_i(\sigma+it)|$ is increasing in $\sigma$ if $\alpha_i>\beta_i$ and decreasing if $\beta_i>\alpha_i$.
Finally, suppose that for each $|t|\leq T_0$ we have 
\[ |f(\sigma + it)| < \prod_i {\rm min}_\sigma( |G(\sigma+it)| )^{\alpha_i\frac{b-\sigma}{b-a} + \beta_i\frac{\sigma-a}{b-a}}.\]
Then for $a \leq \sigma \leq b$, we have
\[|f(\sigma+it)| \leq
\bigg(\prod_{i=1}^r \left|G_i(\sigma+it)\right|^{\alpha_i}   \bigg)^{\frac{b-\sigma}{b-a}}
\bigg(\prod_{i=1}^r \left|G_i(\sigma+it)\right|^{\beta_i}  \bigg)^{\frac{\sigma-a}{b-a}} .
\]
\end{theorem}
\begin{proof}
    The proof is as before, except that to obtain that
    \[ F(s)/H(s) \]
    is bounded on the boundary for $|t|<T_0$ we must use the stronger hypothesis. We note that the hypothesis is strong enough to descend to the inductive cases which need to be considered.
\end{proof}

\begin{remark}
To verify the condition 
\[ |f(\sigma + it)| < \prod_i {\rm min}_\sigma( |G_i(\sigma+it)| )^{\alpha_i\frac{b-\sigma}{b-a} + \beta_i\frac{\sigma-a}{b-a}}\]
one could check numerically on the whole region using complex ball arithmetic, alternatively, one can prove something stronger by using the maximum modulus condition to verify something on only the boundary.

Given that in practice one is interested in large $t$ one can typically arrange for the condition for small values to be satisfied by adjusting $G$ and preserving the asymptotics as one does in selecting $Q_1$ and $Q_2$ in Theorem \ref{thm:phrag2}.
\end{remark}

\section{Errors in the literature}\label{sec:errors}

There are unfortunately a number of errors in the literature where $|\log( Q+ s)|$ was incorrectly handled as $\log |Q+ s|$. Consequently, correcting the literature in some cases would require more than simply checking the hypothesis needed in our corrected version; it would likely also require changes to the numerical results.
 In this section we highlight how one can almost recover the numerics from such an error, at least for $t$ that are not too small (for an illustration of what we mean by almost see Section \ref{sec:examples}).

Moreover, in number theory it is the case that bounds are often proven in the form
\[ |\zeta(1 + it)| < C\log|t|\qquad t>e \]
so that the input we have is actually stronger than the bound $|\log (\sigma+it)|$ which would normally be used as a hypothesis to our theorem. It would be desirable that the conclusions we prove would be similarly strong.
As noted before the discrepancy here is that as $t\to \infty$
\[ |\log (\sigma+it)|  =  (1+O_\sigma(1/(\log |t|)^2)) \log|t|. \]
Unfortunately, the $(1/\log|t|)^2$ term shrinks slowly enough to impact numerics in some applications.  
Noticing that
\[ \tfrac{1}{2}\left|\log (1 - (\sigma+it)^2)\right| = (1+O_\sigma(1/|t|^2)) \log|t| \]
provides the core idea to improve this discrepancy.
The essential idea here is that ${\rm arg}(1+\sigma+it) \sim - {\rm arg}(1-\sigma-it) $ so that we can obtain cancellation of error terms as in the formulas below.

Now, given an interval $[a,b]$ we define functions
\[ G_{Q_1,Q_2}(s) = \frac{1}{2}\log( (Q_1+s)(Q_2-s) )\]
where $Q_1+a\ge 1$ and $Q_2-b\ge 1$.
Notice that we have 
\begin{align*}
\log(Q_1+\sigma+it) 
&= \frac{1}{2} \log ((Q_1+\sigma)^2 + t^2) + i \arctan\left(\frac{t}{Q_1+\sigma}\right) \\
&= \log t + \log\left(1 + \left(\frac{Q_1+\sigma}{t}\right)^2\right)+ i \arctan\left(\frac{t}{Q_1+\sigma}\right)
\end{align*}
and
\begin{align*}
\log(Q_2-\sigma-it) 
&= \frac{1}{2} \log ((Q_2-\sigma)^2 + t^2) - i \arctan\left(\frac{t}{Q_2-\sigma}\right) \\
&= \log t + \log\left(1 + \left(\frac{Q_2-\sigma}{t}\right)^2\right) - i \arctan\left(\frac{t}{Q_2-\sigma}\right)
\end{align*}
so that with
we have\footnote{Using $\arctan(x/y)=\pi/2-\arctan{y/x}$.}
\begin{align*}
G_{Q_1,Q_2}(\sigma+it)
&= \log t + \log\left(1 + \left(\frac{Q_1+\sigma}{t}\right)^2\right)\left(1 + \left(\frac{Q_2-\sigma}{t}\right)^2\right) \\&\quad+ i \left(\arctan\left(\frac{Q_2-\sigma}{t}\right)-\arctan\left(\frac{Q_1+\sigma}{t}\right)\right).
\end{align*}
Approximations for $\log$ and $\arctan$, which can be made precise, will give as $t\to \infty$  that
\[ \log\left(\left(1 + \left(\frac{Q_1+\sigma}{t}\right)^2\right)\left(1 + \left(\frac{Q_2-\sigma}{t}\right)^2\right)\right) \sim \left(\frac{Q_1+\sigma}{t}\right)^2 + \left(\frac{Q_2-\sigma}{t}\right)^2 \]
and
\[ \arctan\left(\frac{Q_2-\sigma}{t}\right)-\arctan\left(\frac{Q_1+\sigma}{t}\right) \sim 
\frac{Q_2-Q_1-2\sigma}{t}.\]
Then for $\sigma\in[a,b]$ we will have, using $|1+i/t|=1+O(1/t^2)$, that 
\[ \log|t| < |G_{Q_1,Q_2}(\sigma+it)| <  \left(1+O_{Q_1,Q_2,\sigma}\left(\frac{1}{t^2\log|t|}\right)\right)\log|t|.\]
Moreover, $|G_{Q_1,Q_2}(\sigma+it)|$ will be decreasing in $\sigma$ for $\sigma\in [-Q_1, (Q_2-Q_1)/2]$ and increasing for $\sigma\in [(Q_2-Q_1)/2,Q_2]$ provided $t$ sufficiently large (depending on $Q_1$ and $Q_2$). Hence with a careful choice of $Q_1$ and $Q_2$ it can satisfy the conditions of Theorem \ref{thm:phrag2}.

For iterated logarithms the situation can now be handled as follows.
Notice that for appropriate $Q_1$ and $Q_2$ we have
\[ \arg(G_{Q_1,Q_2}(\sigma+it)) < O(\tfrac{1}{t}) \]
and hence
\[ | \log (G_{Q_1,Q_2}(\sigma+it))| = \left(1+ O\left(\frac{1}{t^2 (\log t)^2}\right)\right) \log|G_{Q_1,Q_2}(\sigma+it)| \]
provides a reasonable assymptotic approximation to $\log\log|t|$
One can use this directly to introduce $\log\log|t|$ terms while applying Theorem \ref{thm:phrag2}. 

Another potentially useful application comes from observing that functions of the form
\[ H(s) = {\rm exp}( (\beta-\alpha) s \log( G_{Q_1,Q_2}(s)) )G_{Q_1,Q_2}(s)^\alpha \]
will have 
\[ |H(\sigma + it)| = (1+O(1))(\log|t|)^{\alpha(1-\sigma) + \beta\sigma} \]
so that one could potentially directly use it to prove a theorem like Theorem \ref{thm:phrag2}, at the expense of introducing a constant.

\section{Examples}\label{sec:examples}

We provide several examples illustrating how to apply our theorems to obtain bounds for the Riemann zeta function in contexts where the exponent on the log factor in the bound is decreasing in $\sigma$, constant in $\sigma$, or increasing in $\sigma$.

\subsection{First Example}

We have the following bounds on $\zeta(s)$ on the line $1+it$ for $|t|>3$
\begin{align*}
|\zeta(1+it)| & \leq \log|t|\\
|\zeta(1+it)| & \leq \tfrac{1}{2}\log|t| + 1.93\\
|\zeta(1+it)| & \leq \tfrac{1}{5}\log|t| + 44.02\\
|\zeta(1+it)| & \leq 1.731\frac{\log|t|}{\log\log|t|}\\
|\zeta(1+it)| & \leq 58.096(\log|t|)^{2/3}
\end{align*}
which come from \cite{patel_explicit_2022} (first three), \cite{hiary2023explicit} (fourth), and \cite{BELLOTTI2024128249} or \cite{leong2024explicitestimateslogarithmicderivative} (for the final).
Additionally it is well known that for any $\eta>1$ we have
\[ |\zeta(\eta+it)| \leq \zeta(\eta). \]
Since for $\sigma>1$ we have
\[ \left| \frac{\sigma+it-1}{\sigma+it} \right| < 1 \]
all of the above bounds also hold with $\zeta(s)$ replaced by $\frac{s-1}{s}\zeta(s)$.

One may verify numerically, with the maximum modulus principle simplifying the check, that
for $\sigma\in [1,2]$ and $t\in[0,3]$ that 
$\left|\frac{s-1}{s}\zeta(s)\right| \leq 1 $
with the maximum being taken on at $s=1$. Also, for
 $\sigma\in [1,2]$ and $t\in[0,30]$ we have
 $ \left|\frac{s-1}{s}\zeta(s)\right| \leq 2.1 $
 the maximum being near $1+27.7i$.

Now, consider the functions
\begin{align*} 
G_1(s) &= \frac{1}{2}\left(\log(e+s) + \log(e+2-s) \right) \\
G_2(s) &= \frac{1}{4}\left(\log(e+s) + \log(e+2-s) \right) + 1.93 \\
G_3(s) &= \frac{1}{10}\left(\log(e+s) + \log(e+2-s) \right) + 44.02 \\
G_4(s) &= 1.731\frac{\log(e+s) + \log(e+2-s)}{2\log \left(\frac{1}{2}\left(\log(e+s) + \log(e+2-s)\right)\right)} \\
G_5(s) &= 58.096\left(\frac{1}{2}\left(\log(e+s) + \log(e+2-s) \right)\right)^{2/3}.
\end{align*}
We verify numerically that for $\sigma\in [1,2]$ and $t\in[0,3]$ the absolute values of each of these is at least $1$, and for $G_4(s)$ it is bounded below by $3$ for $\sigma\in [1,2]$ and $t\in[0,30]$
Moreover, for $t\ge 3$ each of $G_1$, $G_2$, $G_3$, and $G_5$ are increasing in $\sigma$ and for $t\ge 30$ we have $G_4$ is increasing in $\sigma$.
For $\sigma\in [1,2]$ and $|t|\ge 3$ we have
\begin{align*}
\log|t| &\leq \left|\frac{1}{2}\left(\log(e+s) + \log(e+2-s) \right)\right| \\
\frac{\log|t|}{\log\log|t|} &\leq \left|\frac{\log(e+s) + \log(e+2-s)}{2\log \left(\frac{1}{2}\left(\log(e+s) + \log(e+2-s)\right)\right)} \right| 
\end{align*} 
We may thus apply Theorem \ref{thm:phrag2} to the function $f(s) = \frac{s-1}{s}\zeta(s)$ between $\sigma=1$ and $\sigma=\eta$ for some $\eta\in [1,2]$ (for each $G_i$ separately) where the exponents on $G_i$ are $\alpha=1$ and $\beta=0$ and as in Corollary \ref{cor:simple} we introduce $A=1$ and $B=\zeta(\eta)$.

Notice that for the cases of $G_1$, $G_2$, $G_3$ and $G_5$ the conditions for $|t| \leq 3$ are trivially verified as $\zeta(\eta) \ge \zeta(2)>1$. Similarly with $G_4$ with $|t| \leq 30$ and $\eta<1.5$ as $\zeta(1.5)>2.1$. To extend to $\eta\in[1.5,2]$ one can use that for $\sigma\in[1.5,2]$ and $t\in[0,30]$ we have
$\left| \frac{\sigma+it-1}{\sigma+it}\zeta(s) \right| < \zeta(2)$.
We thus conclude that for any $\eta\in[1,2]$ that for $\sigma={\rm Re}(s) \in[1,\eta]$ we have
\[ \left|{\frac{s-1}{s}\zeta(s)} \right| \leq  \zeta(\eta)^{\frac{\sigma-1}{\eta-1}} |G_i(s)|^{\frac{\eta-\sigma}{\eta-1}}. \]
Now notice
\begin{align*}
    &\frac{\left|\sigma+it\right|}{\left|\sigma-1+it \right|}  |G_1(\sigma+it)|/\log|t|&\\
    &\frac{\left|\sigma+it\right|}{\left|\sigma-1+it\right|}  |G_2(\sigma+it)|/( \tfrac{1}{2}\log|t| + 1.93)&\\
    &\frac{\left|\sigma+it\right|}{\left|\sigma-1+it\right|}  |G_3(\sigma+it)|/( \tfrac{1}{5}\log|t| + 44.02)&\\
    &\frac{\left|\sigma+it\right|}{\left|\sigma-1+it\right|}  |G_4(\sigma+it)|/(1.731\frac{\log|t|}{\log\log|t|})&\\
    &\frac{\left|\sigma+it\right|}{\left|\sigma-1+it\right|}  |G_5(\sigma+it)|/(58.096(\log|t|)^{2/3})&\\
\end{align*}
are all decreasing in $t$ with a limit of $1$. 
This leads to constants $C_i(t_0)$ where
\begin{align}
\left|{\zeta(\sigma+it)}\right|&\leq C_1(t_0)\zeta(\eta)^{\frac{\sigma-1}{\eta-1}}  (\log|t|)^{\frac{\eta-\sigma}{\eta-1}}\\
\left|{\zeta(\sigma+it)}\right|&\leq C_2(t_0)\zeta(\eta)^{\frac{\sigma-1}{\eta-1}} (\tfrac{1}{2}\log|t| + 1.93)^{\frac{\eta-\sigma}{\eta-1}}\\
\left|{\zeta(\sigma+it)}\right|&\leq C_3(t_0)\zeta(\eta)^{\frac{\sigma-1}{\eta-1}} (\tfrac{1}{5}\log|t| + 44.02)^{\frac{\eta-\sigma}{\eta-1}}\\
\left|{\zeta(\sigma+it)}\right|&\leq C_4(t_0)\zeta(\eta)^{\frac{\sigma-1}{\eta-1}} (1.731\frac{\log|t|}{\log\log|t|})^{\frac{\eta-\sigma}{\eta-1}}\\
\left|{\zeta(\sigma+it)}\right|&\leq C_5(t_0)\zeta(\eta)^{\frac{\sigma-1}{\eta-1}} (58.096(\log|t|)^{2/3})^{\frac{\eta-\sigma}{\eta-1}}
\end{align}
for all $t>t_0$ and $\sigma\in[1,\eta]$. One verifies that in all cases we have
\[ 
\begin{array}{ccc}
 C_i(10) < 1 + 4.6\cdot10^{-2} 
 & C_i(10^2) < 1 + 3.2\cdot10^{-4} 
 & C_i(10^3) < 1 + 2.6\cdot10^{-6} \\ 
 C_i(10^4) < 1+2.4\cdot10^{-8}
 &  C_i(10^5) < 1+2.2\cdot10^{-10}
 &  C_i(10^6) <  1+2.1\cdot10^{-12}.
 \end{array}
 \]

\subsection{Second Example}

To bound $|\zeta(s)|$ for $\Re(s)\in[5/7,1]$ we shall also need the bound 
\[ |\zeta(5/7+it)|< 1.546 \log(t) t^{1/14} \qquad t\ge 3\]
due to \cite{YANG2024128124}.

Now set $G(s) = \frac{1}{2}\left(\log(e+s) + \log(e+2-s)\right)$ and notice that for $t\ge 3$ and $0<\sigma<2$ we have
\[ \log(t) < \left| G(\sigma+it) \right|. \]
Now set $H(s) = \zeta(s)\frac{s-1}{sG(s)}$. 
For $\sigma=5/7$ and $t\ge 3$ we have
\[ \left|H(\sigma+it)\right| \leq 1.546|\sigma+it|^{1/14}  \]
and for $\sigma=1$ and $t\ge 3$ we have
\[ \left|H(\sigma+it)\right| \leq 1. \]
We apply our Theorem \ref{thm:phrag} to $H(s)$ to conclude that for $\sigma\in[5/7,1]$ we have
\[ |H(\sigma+it)| < 1.546^{\frac{7}{2}(1-\sigma)}|\sigma+it|^{\frac{1}{4}(1-\sigma)}. \]
It follows that for $\sigma\in[5/7,1]$ we have
\[ |\zeta(\sigma+it)| < 1.546^{\frac{7}{2}(1-\sigma)}|\sigma+it|^{\frac{1}{4}(1-\sigma)}\left|\frac{\sigma+it}{\sigma+it-1}G(\sigma+it)\right|. \]
Finally, for $t>10^{5}$ and  $\sigma\in[5/7,1]$ we have
\[\left|\frac{\sigma+it}{\sigma+it-1}G(\sigma+it)\right| < (1+10^{-10})\log|t| \]
so that for $t>10^{5}$ and  $\sigma\in[5/7,1]$ we have
\begin{equation} 
|\zeta(\sigma+it)| < (1+10^{-10})1.546^{\frac{7}{2}(1-\sigma)}|t|^{\frac{1}{4}(1-\sigma)}\log|t|. \end{equation}

\subsection{Third Example}

To bound $|\zeta(s)|$ for $\Re(s)\in[1/2,5/7]$ we shall also need the bound
\[  |\zeta(\sigma+it)| \leq 66.7t^{\frac{27}{164}}\]
due to \cite{patel2023explicit}.

We let $G(s) = \frac{1}{2}\left(\log(4e+s) + \log(4e+2-s)\right)$ so that $|G(\sigma+it)|$ is decreasing in $\sigma$ for $\sigma\in[1/2,5/7]$ and $t\ge 3$.

We thus have for $t \ge 3$ that
\[ |\zeta(1/2+it)| \leq 66.7t^{\frac{27}{164}} \leq 66.7(1/2 + t)^{\frac{27}{164}}|G(1/2+it)|^0 \]
and
\[ |\zeta(5/7+it)| \leq 1.546|t|^{1/14} \log t \leq 1.546|5/7+it|^{1/14} |G(5/7+it)|. \]

Now, since the minimum value of $1.546|\sigma+it|^{1/14} |G(\sigma+it)|$ and $66.7|\sigma + t|^{\frac{27}{164}}$
where $\sigma\in[1/2,5/7]$ and $t\in [0,3]$ are both larger than the maximum value of $|\zeta(\sigma+it)|$ in the same region we may apply Theorem \ref{thm:phrag2} and conclude
\[
|\zeta(\sigma+it)| \leq \left(66.7|\sigma + it|^{\frac{27}{164}}\right)^{\frac{14}{3}(\frac{5}{7}-\sigma)}\left(1.546|\sigma+it|^{1/14} |G(\sigma+it)| \right)^{\frac{14}{3}(\sigma-\frac{1}{2})} .
\]
As before it is now straight forward to verify that for $t>10^5$ and $\sigma\in[1/2,5/7]$ we have
\begin{equation}  |\zeta(\sigma+it)| \leq (1+10^{-9})66.7^{\frac{14}{3}(\frac{5}{7}-\sigma)}1.546^{\frac{14}{3}(\sigma-\frac{1}{2})} t^{\frac{47}{123}-\frac{107}{246}\sigma} (\log|t|)^{\frac{14}{3}(\sigma-\frac{1}{2})}.
\end{equation}

\printbibliography

\end{document}